\documentclass[a4paper,12pt,reqno]{amsart}
\usepackage{amsmath, amssymb,amsthm}
\usepackage{enumerate}
\usepackage{cite}
\numberwithin{equation}{section}
\newtheorem{theorem}{Theorem}[section]
\newtheorem{lemma}{Lemma}[section]
\newtheorem{proposition}{Proposition}[section]

\theoremstyle{definition}
\newtheorem{remark}{Remark}[section]

\newtheorem{note}{Note}[section]

\begin{document}
\bibliographystyle{amsplain}
\title{{{
Biorthogonal rational functions of $R_{II}$ type
}}}
\author{
Kiran Kumar Behera
}
\address{
Department of Mathematics,
Indian Institute of Technology, Roorkee-247667,
Uttarakhand, India
}
\email{krn.behera@gmail.com}
\author{
A. Swaminathan
}
\address{
Department of  Mathematics  \\
Indian Institute of Technology, Roorkee-247 667,
Uttarakhand,  India
}
\email{swamifma@iitr.ac.in, mathswami@gmail.com}
\bigskip
\begin{abstract}
In this work, a sequence of orthonormal rational functions that is also biorthogonal to another sequence of rational
functions arising from recurrence relations of $R_{II}$ type is constructed.
The biorthogonality is proved by a procedure which we call Zhedanov's method.
A particular case is considered that provides a Christoffel type transformation of the
generalized eigenvalue problem with a reformulation different
from the existing literature.
\end{abstract}

\maketitle

\section{Introduction}
\label{sec: introduction}

Recurrence relations of the form
\begin{align}
\label{eqn: R-II recurrence general form}
\mathcal{P}_{n+1}(z)=\rho_n(z-\nu_n)\mathcal{P}_{n}(z)+
\tau_n(z-a_n)(z-b_n)\mathcal{P}_{n-1}(z),
\quad n\geq1,
\end{align}
with initial conditions $\mathcal{P}_0(z)=1$ and $\mathcal{P}_1(z)=\rho_0(z-\nu_0)$
are studied extensively
\cite{Ismail-Masson-generalized-orthogonality-JAT-1995}
 to define families of biorthogonal functions having explicit representations in terms of basic hypergeometric functions
(see \cite{Rosengren-Rahman-biorthogonal-Constrapprx-2017} for a recent work).
 Further, it was shown \cite{Ismail-Masson-generalized-orthogonality-JAT-1995}
 that if
 \begin{align}
 \label{eqn: conditions for R-II recurrence relation}
 \mathcal{P}_n(a_n)\neq0,
 \quad
 \mathcal{P}_n(b_n)\neq0,
 \quad
 \tau_n\neq0,
 \end{align}
then there exists a rational function
$\phi_n(z)=
\prod_{k=1}^{n}(z-a_k)^{-1}(z-b_k)^{-1}\mathcal{P}_n(z)$
and a linear functional $\mathfrak{M}$ defined on the span
$\{z^k\phi_n(z): 0\leq k\leq n\}$ such that the relation
$\mathfrak{M}(z^k\phi_n(z))=0$, for $0\leq k<n$ holds.
Conversely, one can always obtain
\eqref{eqn: R-II recurrence general form}
from a sequence of rational functions $\{\phi_n(z)\}_{n=0}^{\infty}$
having poles at
$\{a_k\}_{k=1}^{\infty}$ and $\{b_k\}_{k=1}^{\infty}$
and satisfying a three term recurrence relation.
Following
\cite{Ismail-Masson-generalized-orthogonality-JAT-1995} (see also \cite{Ranga-Ismail-R-II-type-recurrences-arxiv}),
we call
\eqref{eqn: R-II recurrence general form}
as recurrence relation of $R_{II}$ type.

Related to such recurrence relations are important concepts of rational functions satisfying both orthogonality and
biorthogonality properties.
The theory of rational functions orthogonal on the unit circle is developed parallel to that of polynomials orthogonal on the
unit circle
and is available in the monograph
\cite{Bultheel-book-ORF}.
A sequence of orthogonormal rational functions is obtained from
the Gram-Schmidt orthonormalization process in the linear
space of rational functions which, in fact, can be characterized by the
poles of the basis elements as well.
In this direction, \cite{Bultheel-poles-on-unit-circle-JMAA-1994,
Li-regularity-ORF-poles-unit-circle-JCAM=1999},
starting from a set of pre-defined poles,
the rational functions are characterized by Favard type theorems
as well as in terms of three-term recurrence
relations similar to that of orthogonal polynomials on the real line \cite{Chihara-book, Ismail-book},
but with rational coefficients.
The effect of poles on the asymptotics of the Christoffel functions associated with the orthogonal rational
functions and their interval of orthogonality
is also studied \cite{Deckers-Lubinsky-poles-Christoffel-JAT-2012}.
For recent generalizations in the theory, see
\cite{Bultheel-Technical-report-2017,
Deckers-CMV-associated-RF-JAT-2011,
Velazquez-spectral-methods-2008} and references therein.

Following~\cite{Konhauser-biorthogonal-JMAA-1965}, two sequences of functions
$\{\mathcal{R}_n(z)\}$ and $\{\mathcal{Q}_n(z)\}$ are said to be biorthogonal,
if they satisfy
\begin{align}
\label{eqn: biorthogonality definition}
\mathfrak{N}(\mathcal{R}_n(z)\mathcal{Q}_m(z))=\kappa_n\delta_{n,m},
\quad \kappa_n\neq0,
\quad n,m\geq0,
\end{align}
with respect to a linear functional $\mathfrak{N}$.
We observe that in contrast to the usual orthogonality condition,
two different sequences are used for the
biorthogonality condition
Further, unlike the case for orthogonal polynomials on the real line \cite{Chihara-book},
the polynomial $\mathcal{P}_n(z)$
satisfying \eqref{eqn: R-II recurrence general form} is the characteristic polynomial of a matrix pencil $G_n-zH_n$,
where both $G_n$ and $H_n$ are tridiagonal matrices
\cite{Spiridonov-Zhedanov-biorthogonality-elliptic-grids-2000,
Spridonov-Zhedanov-spectral-chains-biorthogonality-2000,
Zhedanov-biorthogonal-GEP-JAT-1999}.
\subsection{Motivation for the problem}
\label{sec: motivation}
The components of the eigenvectors of the matrix pencil $G_n-zH_n$ are rational functions with the numerator
polynomials $\mathcal{P}_n(z)$ satisfying
\eqref{eqn: R-II recurrence general form}.
However, these rational functions are not the ones that were used initially to obtain the matrix pencil. In fact,
while the three term recurrence relation satisfied by $\phi_n(z)$ is used to obtain the matrix pencil, the usual process
\cite{Maxim-linear-pencil-JAT-2010}
is to partition the poles to form two new sequences of rational functions
\begin{align}
\label{eqn: rational functions with partition of poles}
p_n^{L}(z)=\dfrac{\mathcal{P}_n(z)}{\prod_{k=1}^{n}(z-a_k)},
\quad
p_n^{R}(z)=\dfrac{\mathcal{P}_n(z)}{\prod_{k=1}^{n}(z-b_k)}
\end{align}
which form the components of the left and right eigenvectors of the matrix pencil $zG-H$.
The two sequences $\{p_n^{L}(z)\}_{n=0}^{\infty}$ and
$\{p_n^{R}(z)\}_{n=0}^{\infty}$
are then used to define two new sequences of rational functions
\cite{Maxim-linear-pencil-JAT-2010,
Derevyagin-Zhedanov-operator-approach-JAT-2009}
satisfying the biorthogonality relation
\eqref{eqn: biorthogonality definition}.
However we note that two sequences of rational functions that are biorthogonal to each other need not themselves form an
orthogonal sequence.

Motivated by the procedure of proving biorthogonality
\cite{Zhedanov-biorthogonal-GEP-JAT-1999}, which we call as
Zhedanov's method,
the central theme of the manuscript is to
study a sequence of rational functions that is both orthogonal as well as biorthogonal.
Precisely, we are interested in constructing a sequence of \emph{orthogonal} rational functions $\{\varphi_n(z)\}$
satisfying the following two properties:
\begin{enumerate}[(i)]

\item The related matrix pencil has the numerator polynomials $\mathcal{P}_n(z)$ as the characteristic polynomials
and $\varphi_n(z)$ as components of the eigenvectors.

\item The orthogonal sequence $\{\varphi_n(z)\}$ is also \emph{biorthogonal} to another sequence of rational functions.
\end{enumerate}
We note that such a system exists in the case of polynomials. For instance, the two polynomials
$\mathcal{R}_n(z;\alpha,\beta)=\,_2F_1(-n,\alpha+\beta+1;2\alpha+1;1-z)$,
$\mathcal{Q}_n(z)=\mathcal{R}_n(z;\alpha,-\beta)$,
$n\geq1$,
were proved to be biorthogonal
\cite{Askey-discussion-Szego-paper-1982}
with respect to the weight function
$\omega(\theta)=(2-2\cos\theta)^\alpha(-e^{i\theta})^{\beta}$, $\theta\in[-\pi,\pi]$, $\rm{Re}\,\alpha>-1/2$.
The sequence $\{\mathcal{R}_n(z;\alpha,\beta)\}_{n=0}^{\infty}$
was later proved to be orthogonal with respect to the weight
$\hat{\omega}(\theta)=2^{2\alpha}e^{(\pi-\theta){\rm{Im}\beta}}\sin^{2\alpha}\theta/2$
if $\alpha\in\mathbb{R}$, $\alpha>-1/2$ and $i\beta\in\mathbb{R}$
\cite{Ranga-szego-polynomials-2010-AMS}.
The present problem serves to find an abstract rational analogue of such cases of orthogonal sequences
satisfying biorthogonality properties as well.
%

The paper is organized as follows.
Section \ref{sec: fundamental spaces and recurrence relations} introduces the fundamental spaces
and the orthogonal rational functions that lead to recurrence relations of $R_{II}$ type.
In Section \ref{sec: recovering rational functions from biorthogonality}, the reverse
procedure, that is, starting with $R_{II}$ recurrences, and recovering the same orthogonal rational functions via
biorthogonality relations is provided. In Section \ref{sec: spectral transformation: the Christoffel case},
the Christoffel type transform of a particular case of our orthogonal rational functions is discussed.
\section{Fundamental spaces and associated rational functions}
\label{sec: fundamental spaces and recurrence relations}
Let $\{\alpha_j\}_{j=1}^{\infty}$ and $\{\beta_j\}_{j=0}^{\infty}$
be two given sequences where $\beta_0:=0$,
\begin{align}
\label{definition of alpha-j and beta-j}
\alpha_j,\beta_j\in\mathbb{C}\setminus\{0\},
\quad
\alpha_j\neq\infty,
\quad j\geq1.
\end{align}
We define
 \begin{align*}
u_{2j}(z):=\dfrac{1}{1-z\bar{\beta}_{j}},
\quad
u_{2j+1}(z):=\dfrac{1}{z-\alpha_{j+1}},
\quad j\geq0.
\end{align*}
The basis $\{u_j\}_{j=0}^{n}$,
$n\geq1$,
generates the linear spaces
$\mathcal{L}_n$=span$\{u_0,u_1,\cdots,u_n\}$ and
$\mathcal{L}=\cup_{n=0}^{\infty}\mathcal{L}_n$.
Equivalently, we also have
$\mathcal{L}_n$=
span$\{\mathfrak{u}_0, \mathfrak{u}_1,\cdots, \mathfrak{u}_n\}$,
where
\begin{align*}
\mathfrak{u}_{2j}(z)=\dfrac{z^{2j}}{\prod_{k=1}^{j}(z-\alpha_k)\prod_{k=1}^{j}(1-z\bar{\beta}_k)},
\quad
\mathfrak{u}_{2j+1}(z)=
\frac{z}{z-\alpha_{j+1}}\mathfrak{u}_{2j}(z),
\quad j\geq0.
\end{align*}
Further, the product spaces
$\mathcal{L}_m\cdot\mathcal{L}_n$
 and
 $\mathcal{L}\cdot\mathcal{L}$
 consist of functions of the form
$h_{m,n}(z)=f_m(z)g_n(z)$ and
$h(z)=f(z)g(z)$ respectively,
where
$f_m(z)\in\mathcal{L}_m$,
$g_n(z)\in\mathcal{L}_n$
and
$f(z), g(z)\in\mathcal{L}$.

The substar transform $h_{\ast}(z)$
of a function $h(z)$
is defined as
$h_{\ast}(z)=\overline{h(1/\bar{z})}$.
Let
$\mathfrak{L}$ be a linear functional defined on
$\mathcal{L}\cdot\mathcal{L}$
such that
\begin{align}
\label{eqn: definition of inner product}
\langle f(z),g(z)\rangle:=
\mathfrak{L}(f(z)g_{\ast}(z)),
\end{align}
is Hermitian and positive-definite, and hence defines an inner product
on the space $\mathcal{L}$.
We note that $\mathfrak{L}$ is said to be Hermitian if it satisfies
$\mathfrak{L}(h_{\ast})=\overline{\mathfrak{L}(\bar{h})}$
for every $h\in\mathcal{L}\cdot\mathcal{L}$
and positive definite if
$\mathfrak{L}(hh_{\ast})>0$
for every
$h\neq0\in\mathcal{L}$.
Let $\varphi_j(z)$, $j\geq0$, be the sequence of functions
that are orthonormal
with respect to $\mathfrak{L}$ and
obtained from the Gram-Schmidt process of the basis
$\{\mathfrak{u}_j\}_{j=0}^{n}$, $n\geq1$.
That is
$\varphi_j(z)$, $j\geq0$, satisfy the orthogonality property
\begin{align*}
\langle\varphi_m(z), \varphi_n(z)\rangle=
\mathfrak{L}(\varphi_m(z)\varphi_{n\ast}(z))=
\delta_{m,n},
\quad m,n=0,1,\cdots.
\end{align*}
Further, it is clear that $\varphi_n(z)$
are rational functions of the form $\varphi_0(z)=1$,
\begin{align}
\label{eqn: form of rational functions}
\begin{split}
\varphi_{2j+2}(z)=\dfrac{r_{2j+2}(z)}{\prod_{k=1}^{j+1}(z-\alpha_k)\prod_{k=1}^{j+1}(1-z\bar{\beta}_k)},
\quad j\geq0,
\\
\varphi_{2j+1}(z)=\dfrac{r_{2j+1}(z)}{\prod_{k=1}^{j+1}(z-\alpha_k)\prod_{k=1}^{j}(1-z\bar{\beta}_k)},
\quad j\geq0,
\end{split}
\end{align}
where $r_n(z)\in\Pi_n$,
the linear space of polynomials of degree at most $n$
Moreover,  $\mathcal{L}_{2n}$ can now be interpreted
as the space of rational functions having poles
belonging to the set
$\{\alpha_1,\cdots,\alpha_n, 1/\bar{\beta}_1,\cdots,1/\bar{\beta}_n\}$
with the order of the pole at $\alpha_j$ or $1/\bar{\beta}_j$
depending on its multiplicity.
The rational function
$\varphi_{2n}(z)\in\mathcal{L}_{2n}$
has a simple pole at each of the points
$\alpha_1,\cdots,\alpha_n, 1/\bar{\beta}_1,\cdots,1/\bar{\beta}_n$.
and $\alpha_j$ and $\beta_j$ are as defined in \eqref{definition of alpha-j and beta-j}.
A similar interpretation for $\mathcal{L}_{2n+1}$
follows.

The regularity conditions in the present case can be obtained as follows.
The expansion in terms of the basis elements gives
\begin{align*}
\varphi_{2n}(z)=A_0+\dfrac{A_1z}{z-\alpha_1}+\dfrac{A_2z^2}{(z-\alpha_1)(1-z\bar{\beta}_1)}+
\cdots+\dfrac{A_{2n}z^{2n}}{\prod_{i=1}^{n}(z-\alpha_i)\prod_{i=1}^{n}(1-z\bar{\beta}_i)},
\end{align*}
so that $r_{2n}(z)=A_{0}\prod_{i=1}^{n}(z-\alpha_i)\prod_{i=1}^{n}(1-z\bar{\beta}_i)+\cdots+A_{2n}$.
Then $A_{2n}\neq0$ if
\begin{align}
\label{regularity condition 1}
r_{2n}(\alpha_n)\neq0
\quad\mbox{and}\quad
r_{2n}(1/\bar{\beta}_n)\neq0.
\end{align}
Similarly, for $\varphi_{2n+1}(z)$, we obtain
\begin{align}
\label{regularity condition 2}
r_{2n+1}(\alpha_{n+1})\neq0
\quad\mbox{and}\quad
r_{2n+1}(1/\bar{\beta}_n)\neq0.
\end{align}
The regularity conditions \eqref{regularity condition 1}
and \eqref{regularity condition 2} are required to guarantee that
$\varphi_{2n}(z)\in\mathcal{L}_{2n}\setminus\mathcal{L}_{2n-1}$
and
$\varphi_{2n+1}(z)\in\mathcal{L}_{2n+1}\setminus\mathcal{L}_{2n}$
respectively.
Using the definition \eqref{eqn: definition of inner product}
of the inner product $\langle\cdot , \cdot\rangle$,
the following result is immediate and will be used in
deriving the recurrence relations for the
orthogonal rational functions $\varphi_j(z)$.
\begin{lemma}
\label{lem: lemma for inner product identities}
Let $\gamma_n\in\mathbb{C}\setminus\{0\}$, $n=1,2,\cdots$.
The following equality
\begin{align*}
\left\langle
\dfrac{1-z\bar{\gamma}_n}{z-\gamma_{n-1}}
f, g
\right\rangle
=
\left\langle
f, \dfrac{z-\gamma_n}{1-z\gamma_{n-1}}g
\right\rangle;
\qquad
\left\langle
\dfrac{z-\gamma_{n+1}}{1-z\bar{\gamma}_n}f,
g
\right\rangle
=
\left\langle
f, \dfrac{1-z\bar{\gamma}_{n+1}}{z-\gamma_n}g
\right\rangle.
\end{align*}
holds for
the rational functions
$f:=f(z)$ and $g:=g(z)$
in $\mathcal{L}$.
\end{lemma}
In addition to the conditions
\eqref{regularity condition 1}
and \eqref{regularity condition 2},
we also assume
$r_{2n}(\beta_{n-1})\neq0$,
$r_{2n}(1/\bar{\alpha}_{n})\neq0$,
$r_{2n+1}(\beta_{n})\neq 0$,
$r_{2n+1}(1/\bar{\alpha}_n)\neq0$.
Here, and in what follows, we consider the sequences $\{\alpha_j\}$
and $\{\beta_j\}$ as defined in \eqref{definition of alpha-j and beta-j},
unless specified otherwise.
\begin{theorem}
\label{thm: recurrence relation for rational functions}
The orthonormal rational functions
$\{\vec{\phi}_n(\lambda)\}_{n=0}^{\infty}$, with
$\vec{\phi}_{-1}(\lambda):=0$ and $\vec{\phi}_0(\lambda):=1$
satisfy the recurrence relations,
\begin{subequations}
\begin{align}
\label{eqn: recurrence relation for phi-2n+1}
\varphi_{2n+1}(z)
&=
\left[
\frac{e_{2n+1}}{z-\alpha_{n+1}}+\frac{d_{2n+1}(z-\beta_{n})}{z-\alpha_{n+1}}
\right]
\varphi_{2n}(z)+
c_{2n+1}\dfrac{1-z\bar{\alpha_n}}{z-\alpha_{n+1}}\varphi_{2n-1}(z),
\\
\varphi_{2n+2}(z)
&=
\left[
\frac{e_{2n+2}}{1-z\bar{\beta}_{n+1}}
+
\frac{d_{2n+2}(1-z\bar{\alpha}_{n+1})}{1-z\bar{\beta}_{n+1}}
\right]
\varphi_{2n+1}(z)+
c_{2n+2}\dfrac{z-\beta_{n}}{1-z\bar{\beta}_{n+1}}
\varphi_{2n}(z),
\label{eqn: recurrence relation for phi-2n}
\end{align}
\end{subequations}
for $n\geq0$,
where $\beta_0:=0$,
the constants $e_j, d_j\in\mathbb{C}$ and
$c_j\in\mathbb{C}\setminus\{0\}$, $j\geq0$.
\end{theorem}
\begin{proof}
Consider the function
\begin{align*}
\mathcal{W}_{2n}(z)=
\dfrac{1-z\bar{\beta}_n}{z-\beta_{n-1}}\varphi_{2n}(z)-
\dfrac{a_{2n}}{z-\beta_{n-1}}\varphi_{2n-1}(z),
\quad n\geq1.
\end{align*}
We first find the appropriate choice of $a_{2n}$ for which
$\mathcal{W}_{2n}(z)\in\mathcal{L}_{2n-1}\setminus\mathcal{L}_{2n-2}$.
Using the rational forms
\eqref{eqn: form of rational functions} of
$\varphi_{2n}(z)$ and
$\varphi_{2n-1}(z)$, we have
\begin{align*}
a_{2n}=\dfrac{r_{2n}(\beta_{n-1})}{r_{2n-1}(\beta_{n-1})}\neq0
\Longrightarrow
\mathcal{W}_{2n}(z)\in\mathcal{L}_{2n-1}\setminus\mathcal{L}_{2n-2}.
\end{align*}
Hence, we can write
\begin{align*}
\mathcal{W}_{2n}(z)=b_{2n}\varphi_{2n-1}(z)+
c_{2n}\varphi_{2n-2}(z)+
\sum_{j=0}^{2n-3}\mathfrak{a}_j^{(2n)}\varphi_j(z),
\end{align*}
where $\mathfrak{a}^{(2n)}_j=
\langle\mathcal{W}_{2n}(z),\varphi_j(z)\rangle$,
$j=0,1,\cdots,2n-3$.
However,
\begin{align*}
\dfrac{z-\beta_n}{1-z\bar{\beta}_{n-1}}\varphi_{j}\in\mathcal{L}_{2n-2}
\quad\mbox{and}\quad
\dfrac{z}{1-z\bar{\beta}_{n-1}}\varphi_{j}\in\mathcal{L}_{2n-2},
\quad
j=0,1,\cdots,2n-3.
\end{align*}
Using Lemma $\ref{lem: lemma for inner product identities}$,
we conclude $\mathfrak{a}_j^{(2n)}=0$ for
$j=0,1,\cdots,2n-3$ and hence
\begin{align*}
\varphi_{2n}(z)=
\left[\dfrac{a_{2n}}{1-z\bar{\beta}_n}+b_{2n}
\dfrac{z-\beta_{n-1}}{1-z\bar{\beta}_n}\right]
\varphi_{2n-1}(z)+
c_{2n}\dfrac{z-\beta_{n-1}}{1-z\bar{\beta}_n}
\varphi_{2n-2}(z),
\quad n\geq1.
\end{align*}
However, we note that both
$\{1, z-\beta_{n-1}\}$ and
$\{1, 1-z\bar{\alpha}_n\}$
form a basis for $\Pi_1$
and hence writing
$a_{2n}+b_{2n}(z-\beta_{n-1})=
e_{2n}+d_{2n}(1-z\bar{\alpha}_n)$,
the recurrence relation
\eqref{eqn: recurrence relation for phi-2n}
follows.
To prove $c_{2n}\neq0$,
we multiply both sides of
\eqref{eqn: recurrence relation for phi-2n} by
$\frac{1-z\bar{\beta}_n}
{\prod_{i=1}^{n}(1-z\bar{\alpha}_i)\prod_{i=1}^{n-1}(z-\beta_i)}$,
so that the definition of the inner product
\eqref{eqn: definition of inner product}
gives
\begin{align*}
c_{2n}
\left\langle
\varphi_{2n-2}(z),
\dfrac{z^{2n-2}}{\prod_{i=1}^{n}(z-\alpha_i)\prod_{i=1}^{n-2}(1-z\bar{\beta}_i)}
\right\rangle
+e_{2n}
\left\langle
\varphi_{2n-1}(z),\mathfrak{u}_{2n-1}(z)
\right\rangle=0,
\end{align*}
which proves $c_{2n}\neq0$, $n\geq1$.

To derive the recurrence relation for $\varphi_{2n+1}(z)$,
consider
\begin{align*}
\mathcal{W}_{2n+1}(z)=
\dfrac{z-\alpha_{n+1}}{1-z\bar{\alpha}_n}\varphi_{2n+1}(z)-
\dfrac{a_{2n+1}}{1-z\bar{\alpha}_n}\varphi_{2n}(z),
\quad n\geq0,
\end{align*}
%
for $a_{2n+1}=
r_{2n+1}(1/\bar{\alpha}_n)/r_{2n}(1/\bar{\alpha}_n)\neq0$.
As in the case for $\varphi_{2n}(z)$, we arrive at
\begin{align*}
\varphi_{2n+1}(z)=
\left[
\dfrac{a_{2n+1}}{z-\alpha_{n+1}}+b_{2n+1}
\dfrac{1-z\bar{\alpha}_n}{z-\alpha_{n+1}}
\right]\varphi_{2n}(z)+
c_{2n+1}
\dfrac{1-z\bar{\alpha}_n}{z-\alpha_{n+1}}\varphi_{2n-1}(z),
\end{align*}
for $n\geq0$,which can also be written as
\eqref{eqn: recurrence relation for phi-2n+1}
since $\{1, 1-z\bar{\alpha}_n\}$
and $\{1,z-\beta_n\}$ both span the linear space
$\Pi_1$.

To prove $c_{2n+1}\neq0$,
we multiply both sides of the recurrence relation
\eqref{eqn: recurrence relation for phi-2n+1}
by
$\frac{(z-\alpha_{n+1})}
{\prod_{i=1}^{n}(1-z\bar{\alpha}_i)\prod_{i=1}^{n}(z-\beta_i)}$.
The inner product
\eqref{eqn: definition of inner product} and
Lemma \ref{lem: lemma for inner product identities}
gives
\begin{align*}
c_{2n+1}\left\langle
\varphi_{2n-1}(z),
\dfrac{z^{2n-1}}{\prod_{i=1}^{n-1}(z-\alpha_i)\prod_{i=1}^{n}(1-z\bar{\beta}_i)}
\right\rangle
+e_{2n+1}
\langle
\varphi_{2n}(z),\mathfrak{u}_{2n}(z)
\rangle
=0,
\end{align*}
from which it follows that $c_{2n+1}\neq0$, $n\geq1$.
\end{proof}
\subsection{$\varphi_{j}(z)$, $j\geq0$, as components of an eigenvector}
The numerator polynomials of orthogonal rational functions
satisfy the recurrence relations of $R_{II}$ type.
Indeed, from
\eqref{eqn: recurrence relation for phi-2n+1} and
\eqref{eqn: recurrence relation for phi-2n},
it can be shown that
\begin{subequations}
\begin{align}
\label{eqn: recurrence relation for p-2n+1}
r_{2n+1}(z)
&=[e_{2n+1}+d_{2n+1}(z-\beta_n)]r_{2n}(z)+
c_{2n+1}(1-z\bar{\alpha}_n)(1-z\bar{\beta}_{n})r_{2n-1}(z),\\
r_{2n+2}(z)
&=[e_{2n+2}+d_{2n+2}(1-z\bar{\alpha}_{n+1})]r_{2n+1}(z)+
c_{2n+2}(z-\alpha_{n+1})(z-\beta_{n})r_{2n}(z),
\label{eqn: recurrence relation for p-2n}
\end{align}
\end{subequations}
for $n\geq0$, where we define $r_0(z):=1$ and
$\beta_0:=0$.
We use \eqref{eqn: recurrence relation for p-2n+1} and
\eqref{eqn: recurrence relation for p-2n} to obtain a generalized eigenvalue problem
such that the zeros of $r_j(z)$, $j\geq1$,
are the eigenvalues
(that is, $r_j(z)$ is the characteristic polynomial)
while the corresponding rational functions
are the components of the corresponding eigenvector.
Consider two infinite matrices
$\mathcal{H}=(h_{i,k})_{i,k\geq 0}^{\infty}$ and
$\mathcal{G}=(g_{i,k})_{i,k\geq 0}^{\infty}$,
where
\begin{align*}
\mathcal{H}=\left(
               \begin{array}{ccccc}
                 d_1        & g_1                         & 0       &0                       & \cdots \\
                 h_{1,0} & -d_2\bar{\alpha_1} &g_2               &0             & \cdots \\
                 0            & h_{2,1}                  & d_3         & g_3                 & \cdots \\
                 0            & 0                             & h_{3,2}    &   -d_4\bar{\alpha_2}             & \cdots \\
                 0            & 0                             & 0              &   h_{4,3}              & \cdots \\
                 \vdots   & \vdots                   & \vdots             &  \vdots       & \ddots \\
               \end{array}
             \right),
\end{align*}
\begin{align*}
\mathcal{G}=\left(
               \begin{array}{ccccc}
                 -e_1+\beta_0 d_1 & \alpha_1 g_1 & 0  &0 & \cdots \\
                 h_{1,0}\beta_0    & -e_2-d_2 & \bar{\alpha}_1 g_2 &0  & \cdots\\
                 0 & h_{2,1}/\bar{\alpha}_1 & -e_3+\beta_1d_3
                 &\alpha_2g_3      &\cdots\\
                 0 & 0 & h_{3,2}\beta_1  &-e_4-d_4 & \cdots \\
                 0 & 0 & 0  &h_{4,3}/\bar{\alpha}_2 & \cdots \\
                 \vdots & \vdots &\vdots &\vdots   & \ddots \\
               \end{array}
             \right),
\end{align*}
with
$g_{2k+2}=-c_{2k+3}\bar{\beta}_{k+1}/h_{2k+2,2k+1}$,
$g_{2k+1}=-c_{2k+2}/h_{2k+1,2k}$, $k\geq0$.
Here, $\alpha_j$, $\beta_j$, $e_j$, $d_j$ and $c_j$
are the constants
appearing in the recurrence relations
\eqref{eqn: recurrence relation for p-2n+1} and
\eqref{eqn: recurrence relation for p-2n}
while $\{h_{i,i-1}\}_{i=1}^{\infty}$
is a sequence of arbitrary non-vanishing complex numbers.
\begin{proposition}{\rm\cite{Zhedanov-biorthogonal-GEP-JAT-1999}}
\label{thm: numerator polynomials as charecteristic of GEP}
Let $\mathcal{H}_j$ and $\mathcal{G}_j$
denote the $j^{th}$ principal minors of $\mathcal{H}$
and $\mathcal{G}$ respectively.
Then $(-1)^jr_j(\lambda)$, $j\geq1$, is the characteristic polynomial
of the generalized eigenvalue problem
\begin{align}
\label{eqn: generalised eigenvalue problem}
\mathcal{G}_j\vec{\varrho}_j=\lambda \mathcal{H}_j\vec{\varrho}_j,
\end{align}
where $\{r_j\}$ satisfies \eqref{eqn: recurrence relation for p-2n+1} and
\eqref{eqn: recurrence relation for p-2n}.
\end{proposition}
The generalized eigenvalue problem
\eqref{eqn: generalised eigenvalue problem}
has $j-1$ free variables $h_{i,i-1}$
which shows that the matrix pencil associated
with the recurrence relations of $R_{II}$ type is not unique.
We now assign appropriate values to these
free variables to obtain an eigenvector $\vec{\varrho}_j$.
\begin{theorem}
\label{thm: first missing link}
Let the terms of the sequence $\{h_{i, i-1}\}_{i=1}^{\infty}$ be assigned the values
\begin{align*}
h_{2i,2i-1}=-c_{2i+1}\bar{\alpha}_i,
\quad
h_{2i-1,2i-2}=c_{2i},
\quad i\geq1.
\end{align*}
Then,
$\vec{\varrho}_j=
\left(
\begin{array}{cccc}
\varphi_0 & \varphi_1 & \cdots & \varphi_j \\
\end{array}
\right)^{T}
$
is the eigenvector of the generalized eigenvalue problem
\eqref{eqn: generalised eigenvalue problem}
corresponding to the eigenvalue which is a zero of $r_j(\lambda)$.
\end{theorem}
\begin{proof}
Upon substitution of the values of $h_{i,i-1}$,
the recurrence relations
\eqref{eqn: recurrence relation for phi-2n+1} and
\eqref{eqn: recurrence relation for phi-2n}
can be written as
$(-e_1+d_1\beta_0)\varphi_0-\alpha_1\varphi_1=
z[d_1\varphi_0-\varphi_1]$
and
\begin{align*}
-c_{2k+3}\varphi_{2k+1}&-(e_{2k+3}-d_{2k+3}\beta_{k+1})\varphi_{2k+2}-\alpha_{k+2}\varphi_{2k+3}\\
&=z[-c_{2k+3}\bar{\alpha}_{k+1}\varphi_{2k+1}+d_{2k+3}\varphi_{2k+2}-\varphi_{2k+3}],\\
%
\beta_kc_{2k+2}\varphi_{2k}&-(e_{2k+2}+d_{2k+2})\varphi_{2k+1}+\varphi_{2k+2}\\
&=
z[c_{2k+2}\varphi_{2k}-d_{2k+2}\bar{\alpha}_{k+1}\varphi_{2k+1}+\bar{\beta}_{k+1}\varphi_{2k+2}],
\end{align*}
for $k\geq0$, which can be rearranged to yield the matrix equations
\begin{align*}
\mathcal{G}_{2n}\vec{\varrho}_{2n}&=z\mathcal{H}_{2n}\vec{\varrho}_{2n}-(z-\beta_n)\varphi_{2n}\vec{e}_{2n},\\
\mathcal{G}_{2n+1}\vec{\varrho}_{2n+1}&=z\mathcal{H}_{2n+1}\vec{\varrho}_{2n+1}-
(z-\alpha_{n+1})\varphi_{2n+1}\vec{e}_{2n+1},
\end{align*}
where $\vec{e}_j$ is the $j^{th}$ column of the unit matrix.
Observing the fact that $(z-\beta_n)\varphi_{2n}$
does not vanish for $z=\beta_{n}$,
$\vec{\varrho}_{2j}$ becomes an eigenvector for the
generalized eigenvalue problem
\eqref{eqn: generalised eigenvalue problem}
with the zeros of $r_{2n}(z)$ as eigenvalues.
Similarly, $\vec{\varrho}_{2j+1}$ becomes an eigenvector
with the zeros of $r_{2n+1}(z)$
as eigenvalues and the proof is complete.
\end{proof}
Theorems~\ref{thm: recurrence relation for rational functions} and
\ref{thm: first missing link} serve the first step of our construction.
That is, we have obtained a sequence of rational
functions that is orthogonal with respect to the
linear functional $\mathfrak{L}$.
These rational functions are also the components
of the eigenvector of a matrix pencil
whose characteristic polynomials are the
numerator polynomials of such rational functions.
In the next section, we will discuss the
biorthogonality properties of $\{\varphi_n(z)\}$.
\section{A biorthogonality relation for the rational functions}
\label{sec: recovering rational functions from biorthogonality}
In the present section, we use the recurrence relations
\eqref{eqn: recurrence relation for p-2n+1} and
\eqref{eqn: recurrence relation for p-2n}
obtained in Section~\ref{sec: fundamental spaces and recurrence relations}
to define biorthogonality relations involving the
orthogonal rational functions $\{\varphi_j\}$.
To start with, we introduce the rational functions
$\mathcal{O}_0(z)=1$ and
\begin{align}
\label{eqn: intermediary rational functions form}
\begin{split}
\mathcal{O}_{2n+1}(z)
&=
\frac{r_{2n+1}(z)}
{\prod_{j=1}^{n+1}(z-\alpha_j)\prod_{j=1}^{n}(1-z\bar{\alpha}_j)
\prod_{j=0}^{n}(z-\beta_j)\prod_{j=1}^{n}(1-z\bar{\beta}_j)},
\\
\mathcal{O}_{2n+2}(z)
&=
\frac{r_{2n+2}(z)}
{\prod_{j=1}^{n+1}(z-\alpha_j)\prod_{j=1}^{n+1}(1-z\bar{\alpha}_j)
\prod_{j=0}^{n}(z-\beta_j)\prod_{j=1}^{n+1}(1-z\bar{\beta}_j)}.
\end{split}
\end{align}
for $n\geq0$. Here $\{r_j\}$ satisfies
\eqref{eqn: recurrence relation for p-2n+1} and
\eqref{eqn: recurrence relation for p-2n} so that
the sequence $\{\mathcal{O}_j(z)\}$ satisfies
\begin{align*}
(z-\alpha_{n+1})(z-\beta_n)\mathcal{O}_{2n+1}(z)
&=
[e_{2n+1}+d_{2n+1}(z-\beta_n]\mathcal{O}_{2n}(z)+
c_{2n+1}\mathcal{O}_{2n-1}(z),\\
(1-z\bar{\alpha}_n)(1-z\bar{\beta}_n)\mathcal{O}_{2n}(z)
&=
[e_{2n}+d_{2n}(1-z\bar{\alpha}_n)]\mathcal{O}_{2n-1}(z)+
c_{2n}\mathcal{O}_{2n-2}(z),
\end{align*}
for $n\geq1$. Then, similar to Theorem 3.5
and its following corollary of
Ismail and Masson~\cite{Ismail-Masson-generalized-orthogonality-JAT-1995},
we have
\begin{theorem}
\label{thm: orthogonality result similar to Ismail}
Consider the rational functions
given by \eqref{eqn: intermediary rational functions form}.
Then there exists a linear functional $\mathfrak{N}$ on the span of
rational functions $\{z\mathcal{O}_n(z)\}$
such that the orthogonality relation
\begin{align*}
\mathfrak{N}(z^k\mathcal{O}_n(z))=0,
\quad k=0,1,\cdots,n-1,
\end{align*}
holds. Further, if $\mathfrak{N}(1)=m_0$,
$\mathfrak{N}(z^n\mathcal{O}_n(z))=m_n$,
$n\geq1$, then
\begin{align}
\label{eqn: recurrence relation for m-n}
\begin{split}
\bar{\alpha}_n\bar{\beta}_{n}m_{2n}+d_{2n}\bar{\alpha}_{n}m_{2n-1}-c_{2n}m_{2n-2}&=0,
\quad n\geq1\\
m_{2n+1}-d_{2n+1}m_{2n}-c_{2n+1}m_{2n-1}&=0,
\quad n\geq1.
\end{split}
\end{align}
\end{theorem}
We also need the following relations among the
leading coefficients
$r_j(z)$, $j\geq1$.
If $r_j=\kappa_jz^j+\hbox{lower order terms}$,
then from
\eqref{eqn: recurrence relation for p-2n+1} and
\eqref{eqn: recurrence relation for p-2n},
\begin{align}
\label{eqn: recurrence for leading coefficients}
\begin{split}
\kappa_{2n}+d_{2n}\bar{\alpha}_n\kappa_{2n-1}-c_{2n}\kappa_{2n-2}&=0
\quad n\geq1,
\\
\kappa_{2n+1}-d_{2n+1}\kappa_{2n}-\bar{\alpha}_n\bar{\beta}_n c_{2n+1}\kappa_{2n-1}&=0
\quad n\geq1.
\end{split}
\end{align}
It is clear that each of the the recurrence relations \eqref{eqn: recurrence relation for m-n}
and \eqref{eqn: recurrence for leading coefficients} involve two arbitrary initial values.
We choose $m_{0}$ and $m_{1}$ such that $m_1\neq d_1m_0$. \
Since $\kappa_0=1$ and $\kappa_1=d_1$, this implies
$\kappa_0m_1-\kappa_1m_0\neq0$.

Consider another sequence of rational functions
$\{\tilde{\varphi}_j(z)\}_{j=0}^{\infty}$ where
$\tilde{\varphi}_{0}(z):=1$,
\begin{align}
\label{eqn: form of rational functions tilde}
\begin{split}
\tilde{\varphi}_{2n+1}(z)
&=
\frac{r_{2n+1}(z)}{\prod_{j=1}^{n}(1-z\bar{\alpha}_j)\prod_{j=0}^{n}(z-\beta_j)}
\quad\mbox{and}
\\
\tilde{\varphi}_{2n+2}(z)
&=
\frac{r_{2n+2}(z)}{\prod_{j=1}^{n+1}(1-z\bar{\alpha}_j)\prod_{j=0}^{n}(z-\beta_j)},
\end{split}
\end{align}
for $n\geq0$. Here $\{r_j(z)\}$ satisfy
\eqref{eqn: recurrence relation for p-2n+1} and
\eqref{eqn: recurrence relation for p-2n}.
Let $\tilde{\mathcal{J}}_m(z)=
\chi_m^{-1}\tilde{\varphi}_m(z)$, where
$\chi_{2m}=\bar{\alpha}_1(\bar{\beta}_1)^{-1}\cdots\bar{\alpha}_m(\bar{\beta}_m)^{-1}$ and
$\chi_{2m+1}=\bar{\alpha}_1(\bar{\beta}_1)^{-1}\cdots\bar{\alpha}_m(\bar{\beta}_m)^{-1}\bar{\alpha}_{m+1}$.
Define
\begin{align*}
\tilde{\psi}_{2j}(z)&:=\frac{c_{2j+1}(\bar{\beta}_j)^2}{\bar{\alpha}_{j+1}}
\tilde{\mathcal{J}}_{2j-1}(z)-
\frac{d_{2j+1}}{\bar{\alpha}_{j+1}}
\tilde{\mathcal{J}}_{2j}(z)+
\tilde{\mathcal{J}}_{2n+1}(z),
\quad n\geq1,\\
\tilde{\psi}_{2j+1}(z)&:=\frac{c_{2j+2}\bar{\beta}_{j+1}}{\bar{\alpha}_{j+1}}
\tilde{\mathcal{J}}_{2j}(z)-
d_{2j+2}\bar{\alpha}_{j+1}\bar{\beta}_{j+1}
\tilde{\mathcal{J}}_{2j+1}(z)+
\bar{\alpha}_{j+1}
\tilde{\mathcal{J}}_{2j+2}(z),
\quad n\geq0,
\end{align*}
with $\tilde{\psi}_0(z):=1$. The following theorem gives the biorthogonality
relations for $\varphi(z)$ constructed in the previous section.
\begin{theorem}
\label{thm: biorthogonality for phi-2n and phi-2n+1}
The sequences of rational functions $\{\varphi_j(z)\}$ and $\{\tilde{\psi}_j(z)\}$
satisfy the following biorthogonality relations
\begin{align}
\label{eqn: biorthogonality for phi-2n}
\mathfrak{N}(\varphi_{2n}(z)\cdot\tilde{\psi}_m(z))
&=
\frac{c_2c_3\cdots c_{2n+1}
(m_1\kappa_0-m_0\kappa_1)}{\chi_{2n+1}}
\delta_{2n,m},
\\
\mathfrak{N}(\varphi_{2n+1}(z)\cdot\tilde{\psi}_m(z))
&=
\frac{c_2c_3\cdots c_{2n+2}(m_1\kappa_0-m_0\kappa_1)}
{\chi_{2n+2}}
\delta_{2n+1,m},
\label{eqn: biorthogonality for phi-2n+1}
\end{align}
where $m_j=\mathfrak{N}(z^jO_j(z))$ and
$\kappa_j$ is the leading coefficient of $r_j(z)$.
\end{theorem}
\begin{proof}
For simplicity, we write $\varphi_j:=\varphi_{j}(z)$
and similar notations follow for others.
We divide the proof into the following cases.
First, let $m<2n$ and $m$ has even value, say $m=2j$.
Then
\begin{align*}
\mathfrak{N}(\varphi_{2n}\cdot\tilde{\psi}_m)
=
\frac{c_{2j+1}\bar{\beta}_j}{\bar{\alpha}_{j+1}}
\mathfrak{N}(\varphi_{2n}\cdot\tilde{J}_{2j-1})
-
\frac{d_{2j+1}}{\bar{\alpha}_{j+1}}
\mathfrak{N}(\varphi_{2n}\cdot\tilde{J}_{2j})
+
\mathfrak{N}(\varphi_{2n}\cdot\tilde{J}_{2j+1}).
\end{align*}
We evaluate the first term.
We have $\mathfrak{N}(\varphi_{2n}\cdot\tilde{J}_{2j-1})$
\begin{align*}
&=
\frac{1}{\chi_{2j-1}}
\mathfrak{N}
\left(\frac{r_{2n}}{\prod_{k=1}^{n}(z-\alpha_k)\prod_{k=1}^{n}(1-z\bar{\beta}_k)}
\cdot
\frac{r_{2j-1}}{\prod_{k=1}^{j-1}(1-z\bar{\alpha}_k)\prod_{k=0}^{j-1}(z-\beta_k)}
\right)\\
&=
\frac{1}{\chi_{2j-1}}\mathfrak{N}
(\mathcal{O}_{2n}\cdot r_{2j-1}(1-z\bar{\alpha}_{j})\cdots (1-z\bar{\alpha}_{n})
(z-\beta_j)\cdots(z-\beta_{n-1}))\\
&=
\frac{(-\bar{\alpha}_j)\cdots (-\bar{\alpha}_n)\kappa_{2j-1}}{\chi_{2j-1}}m_{2n}.
\end{align*}
A similar evaluation of the remaining two terms yields
\begin{align*}
\mathfrak{N}(\varphi_{2n}\cdot\tilde{J}_{2j})
&=
\frac{(-\bar{\alpha}_{j+1})\cdots (-\bar{\alpha}_{n})\kappa_{2j}}{\chi_{2j}}m_{2n},
\\
\mathfrak{N}(\varphi_{2n}\cdot\tilde{J}_{2j+1})
&=
\frac{(-\bar{\alpha}_{j+1})\cdots (-\bar{\alpha}_{n})\kappa_{2j+1}}{\chi_{2j+1}}m_{2n}.
\end{align*}
Using the relations \eqref{eqn: recurrence for leading coefficients},
we obtain
$\mathfrak{N}(\varphi_{2n}(z)\cdot\tilde{\psi}_m(z))=0$
for $m=2j<2n$.

In the second case, let $m>2n$ and $m$
has odd value, say $m=2j+1$. Then
\begin{align*}
&\mathfrak{N}(\varphi_{2n}\cdot\tilde{\psi}_m)
\\
&=
\frac{c_{2j+2}\bar{\beta}_{j+1}}{\bar{\alpha}_{j+1}}
\mathfrak{N}(\varphi_{2n}\cdot\tilde{\mathcal{J}}_{2j})
-
d_{2j+2}\bar{\alpha}_{j+1}\bar{\beta}_{j+1}
\mathfrak{N}(\varphi_{2n}\cdot\tilde{\mathcal{J}}_{2j+1})
+
\bar{\alpha}_{j+1}
\mathfrak{N}(\varphi_{2n}\cdot\tilde{\mathcal{J}}_{2j+2}),
\end{align*}
so that, as in the case of $\tilde{\psi}_{2j}(z)$, we have
\begin{align*}
\mathfrak{N}(\varphi_{2n}\cdot\tilde{\mathcal{J}}_{2j+2})
&=\frac{\kappa_{2n}m_{2j+2}}{\chi_{2j+2}},
\quad
\mathfrak{N}(\varphi_{2n}(z)\cdot\tilde{\mathcal{J}}_{2j}(z))
=
\frac{\kappa_{2n}m_{2j}}{\chi_{2j}},
\\
\mathfrak{N}(\varphi_{2n}(z)\cdot\tilde{\mathcal{J}}_{2j+1}(z))
&=
\frac{\kappa_{2n}m_{2j+1}}{\chi_{2j+1}}.
\end{align*}
Hence, using \eqref{eqn: recurrence relation for m-n} we have
$\mathfrak{N}(\varphi_{2n}(z)\cdot\tilde{\psi}_m(z))=
0$ for $m=2j+1>2n$.

In the third case, we prove the biorthogonality relations
\eqref{eqn: biorthogonality for phi-2n} and
\eqref{eqn: biorthogonality for phi-2n+1}.
For $m=2n$, we obtain
\begin{align*}
\mathfrak{N}(\varphi_{2n}(z)\cdot\tilde{\psi}_{2n}(z))=
\frac{1}
{\chi_{2n+1}}(\kappa_{2n}m_{2n+1}-d_{2n+1}\kappa_{2n}m_{2n}-
c_{2n+1}\bar{\beta}_n\bar{\alpha}_n\kappa_{2n-1}m_{2n}).
\end{align*}
From \eqref{eqn: recurrence relation for m-n}, we find that
$m_{2n+1}\kappa_{2n}-d_{2n+1}\kappa_{2n}m_{2n}=
c_{2n+1}m_{2n-1}\kappa_{2n}$,
so that
\begin{align*}
\mathfrak{M}(\varphi_{2n}(z)\cdot\tilde{\psi}_{2n}(z))=
\frac{c_{2n+1}}{\chi_{2n+1}}(\kappa_{2n}m_{2n-1}-\bar{\alpha}_n\bar{\beta}_n\kappa_{2n-1}m_{2n}).
\end{align*}
To simplify the numerator in the right hand side above,
we note from \eqref{eqn: recurrence relation for m-n} and
\eqref{eqn: recurrence for leading coefficients}
that the following relations
\begin{align}
\begin{split}
\kappa_{2n}m_{2n-1}-\bar{\alpha}_n\bar{\beta}_n\kappa_{2n-1}m_{2n}
&=
c_{2n}(m_{2n-1}\kappa_{2n-2}-m_{2n-2}\kappa_{2n-1}),
\\
\kappa_{2n-2}m_{2n-1}-\kappa_{2n-1}m_{2n-2}
&=
c_{2n-1}(m_{2n-3}\kappa_{2n-2}-\bar{\alpha}_{n-1}\bar{\beta}_{n-1}m_{2n-2}\kappa_{2n-3}),
\label{eqn: relation between kappa-n and m-n in biorthogonality proof}
\end{split}
\end{align}
hold which further imply that
\begin{align*}
\kappa_{2n}m_{2n-1}-\bar{\alpha}_n\bar{\beta}_n\kappa_{2n-1}m_{2n}=
c_{2n}c_{2n-1}\cdots c_2(m_1\kappa_0-m_0\kappa_1)
\neq0.
\end{align*}
The proof of
\eqref{eqn: biorthogonality for phi-2n+1}
follows the exact techniques and line of argument as in the proof of
\eqref{eqn: biorthogonality for phi-2n}.
Indeed, proceeding as above we obtain, for $m=2n+1$,
\begin{align*}
\mathfrak{N}(\varphi_{2n+1}(z)\cdot\tilde{\psi}_{2n+1}(z))=
\frac{c_{2n+2}(\kappa_{2n}m_{2n+2}-\kappa_{2n+1}m_{2n})}
{\chi_{2n+2}}.
\end{align*}
Simplifying the numerator in the right hand side above, we note from
\eqref{eqn: relation between kappa-n and m-n in biorthogonality proof}
that
\begin{align*}
m_{2n+1}\kappa_{2n}-\kappa_{2n+1}m_{2n}=
c_{2n+1}c_{2n}\cdots c_2(\kappa_0m_1-m_0\kappa_1)
\neq0.
\end{align*}
The proof of the biorthogonality relations
\eqref{eqn: biorthogonality for phi-2n}
and
\eqref{eqn: biorthogonality for phi-2n+1}
for the remaining cases, that is,
$m>2n$, $m=2j$ and $m<2n$, $m=2j+1$,
can be obtained with similar arguments,
thus completing the proof.
\end{proof}
\begin{remark}
The technique of using the leading coefficients $\kappa_n$ and the normalization constants $m_n$ to prove
biorthogonality, as is evident in the present section, is available in the literature,
for example, in Zhedanov~{\rm\cite{Zhedanov-biorthogonal-GEP-JAT-1999}}.
However, the difference between the present work and
Zhedanov~{\rm\cite{Zhedanov-biorthogonal-GEP-JAT-1999}}
is our second objective of proving biorthogonality for exactly the
same rational functions that were used to arrive at the
recurrence relations of $R_{II}$ type for the numerator polynomals
$r_j(z)$ which is also evident from
Remark \ref{remark: difference in Christoffel type transform}.
\end{remark}
\section{Spectral transformation of Christoffel type}
\label{sec: spectral transformation: the Christoffel case}
The Christoffel transformation of well-known orthogonal polynomials is
abundant in the literature \cite[p.~35]{Chihara-book}, \cite[Section.~2.7]{Ismail-book}
\cite{Zhedanov-rational-spectral-transformations-JCAM-1997}.
In the present section, we find a Christoffel type transformation
of the orthogonal rational functions given in
\eqref{eqn: form of rational functions}
for the special case $|\beta_j|=1$ and
$\alpha_j=\alpha\in\mathbb{C}\setminus\{0\}$, $j\geq1$.
We begin with the recurrence relations
\eqref{eqn: recurrence relation for p-2n+1} and
\eqref{eqn: recurrence relation for p-2n}
of $R_{II}$ type
for the numerator polynomials
$\{r_n(z)\}_{n=0}^{\infty}$
which are now written, for $n\geq0$, as
\begin{subequations}
\begin{align}
\label{eqn: recurrence relation for p-2n+1-rho-nu-tao}
r_{2n+1}(z)=\rho_{2n}(z-\nu_{2n})r_{2n}(z)-
\tau_{2n}(z-1/\bar{\alpha})(z-\beta_n)r_{2n-1}(z),
\\
r_{2n+2}(z)=\rho_{2n+1}(z-\nu_{2n+1})r_{2n+1}(z)-
\tau_{2n+1}(z-\alpha)(z-\beta_{n})r_{2n}(z),
\label{eqn: recurrence relation for p-2n-rho-nu-tao}
\end{align}
\end{subequations}
where the new parameters
$\{\rho_n\}$ and $\{\nu_n\}$
are given by
\begin{align*}
\rho_{2n}&=d_{2n+1}, \quad
\nu_{2n}=(d_{2n+1}\beta_n-e_{2n+1})/d_{2n+1},\quad
\tau_{2n}=-c_{2n+1}\bar{\alpha}\bar{\beta}_n,\\
\rho_{2n+1}&=-d_{2n+2}\bar{\alpha}, \quad
\nu_{2n+1}=(e_{2n+2}+d_{2n+2})/(d_{2n+2}\bar{\alpha}),\quad
\tau_{2n+1}=c_{2n+2}.
\end{align*}
The recurrence relations
\eqref{eqn: recurrence relation for p-2n-rho-nu-tao}
and
\eqref{eqn: recurrence relation for p-2n+1-rho-nu-tao}
written in terms of the rational functions
$\varphi_j(z)$, $j\geq0$ (as defined in
\eqref{eqn: form of rational functions})
yield
\begin{align}
\label{eqn: rational functions in Christoffel section}
\begin{split}
(z-\alpha)\varphi_{2n+1}(z)
&=
u_{2n}(z-\nu_{2n})\varphi_{2n}(z)+
\lambda_{2n}(z-1/\bar{\alpha})\varphi_{2n-1}(z),
\\
(z-\beta_{n+1})\varphi_{2n+2}(z)
&=
u_{2n+1}(z-\nu_{2n+1})\varphi_{2n+1}(z)+
\lambda_{2n+1}(z-\alpha_{n})\varphi_{2n}(z),
\end{split}
\end{align}
Moreover, for $n\geq0$,
if we define the shift operators
$\Gamma$ and $\Lambda$ as
\begin{align}
\label{eqn: shift operators Gamma and Lambda for orginal GEP}
\begin{split}
\Gamma\varphi_{2n+1}
&:=\beta_{n+1}\varphi_{2n+2}-
u_{2n+1}\nu_{2n+1}\varphi_{2n+1}-\lambda_{2n+1}\beta_n\varphi_{2n},\\
\Gamma\varphi_{2n}
&:=\alpha\varphi_{2n+1}-
u_{2n}\nu_{2n}\varphi_{2n}-\lambda_{2n}/\bar{\alpha}\varphi_{2n-1},\\
\Lambda\varphi_{2n+1}
&:=\varphi_{2n+2}-u_{2n+1}\varphi_{2n+1}-
\lambda_{2n+1}\varphi_{2n},\\
\Lambda\varphi_{2n}
&:=\varphi_{2n+1}-u_{2n}\varphi_{2n}-\lambda_{2n}\varphi_{2n-1},
\end{split}
\end{align}
then
\eqref{eqn: rational functions in Christoffel section}
leads to the
generalized eigenvalue problem
$\Gamma\vec{\varrho}=z\Lambda\vec{\varrho}$
with the eigenvalue $z$
and the eigenvector
$
\vec{\varrho}=
\left(
\begin{array}{cccc}
\varphi_0 & \varphi_1 & \varphi_2 & \cdots \\
\end{array}
\right)^T.
$
Let $\hat{\varphi}_{2n+1}(z)$ denote the Christoffel type transform of
$\varphi_{2n+1}(z)$, $n\geq0$, obtained under the action of the
operator $\mathfrak{D}$, where
$\mathfrak{D}\varphi_{j}(z)=\hat{\varphi}_{j}(z)$.
We note that $\hat{\varphi}_{2n}(z)$, is an arbitrary rational function
in the present case.
Further, we suppose that
$\mathfrak{D}\vec{\varrho}:=\vec{\hat{\varrho}}$, where
$\vec{\hat{\varrho}}=
\left(
\begin{array}{cccc}
\hat{\varphi}_0 & \hat{\varphi}_1 & \hat{\varphi}_2 & \cdots \\
\end{array}
\right)^{T}$.
The following lemma gives information on the action of the operator
$\mathfrak{D}$ on an arbitrary rational function
$\mathcal{Y}_k:=\mathcal{Y}_k(\lambda)$
which belongs to the space $\mathcal{L}_j$.
\begin{lemma}
\label{lem: Lemma for operator D}
Let
$\mathfrak{D}\mathcal{Y}_k:=
\Omega(z)(\mathcal{Y}_{k+1}+\zeta_j\mathcal{Y}_k)$,
$\mathcal{Y}_k\in\mathcal{L}_j$
for $j\geq0$,
where $\Omega(z)$ is a function of $z$ but independent of $k$ and
hence, is a constant with respect to
$\mathfrak{D}$. Then
\begin{align*}
\zeta_{2j+1}=-\frac{\theta_{2j+2}}{\theta_{2j+1}}
\qquad\mbox{and}\qquad
\zeta_{2j}=-\frac{\theta_{2j+1}}{\theta_{2j}},
\qquad j\geq0,
\end{align*}
where $\theta_j$ is any function satisfying the recurrence relations
\eqref{eqn: rational functions in Christoffel section}.
\end{lemma}
\begin{proof}
Define another operator $\mathfrak{K}$ as
\begin{align}
\label{eqn: operator relation K and D phi-2n+1}
\mathfrak{K}\Gamma=\Gamma^{o}\mathfrak{D}
\quad\mbox{and}\quad
\mathfrak{K}\Lambda=\Lambda^{o}\mathfrak{D}.
\end{align}
Then, the effect of $\mathfrak{K}$ on the generalized eigenvalue problem
$\Gamma\vec{\varrho}=z\Lambda\vec{\varrho}$ yields
$\Gamma^{o}\vec{\hat{\varrho}}=z\Lambda^{o}\vec{\hat{\varrho}}$
which gives the generalized eigenvalue
problem for $\vec{\hat{\varrho}}$.
Further, similar to
\eqref{eqn: shift operators Gamma and Lambda for orginal GEP},
we define the shift operator
$\Gamma^{o}$ by
\begin{align}
\begin{split}
\label{eqn: operator S hat definition phi-2n+1 case}
\Gamma^{o}\mathcal{Y}_{2n}
&:=
\hat{\beta}_{n}\mathcal{Y}_{2n+1}-
\hat{u}_{2n+1}\hat{\nu}_{2n+1}\mathcal{Y}_{2n}-
\hat{\beta}_{n-2}\hat{\lambda}_{2n+1}\mathcal{Y}_{2n-1},
\\
\Gamma^{o}\mathcal{Y}_{2n+1}
&:=
\hat{\alpha}\mathcal{Y}_{2n+2}-
\hat{u}_{2n}\hat{\nu}_{2n}\mathcal{Y}_{2n+1}-
\hat{\lambda}_{2n}/\hat{\bar{\alpha}}\mathcal{Y}_{2n}
\end{split}
\end{align}
and the shift operator $\Lambda^{o}$ by
\begin{align}
\begin{split}
\label{eqn: operator T hat definition-2n+1 case}
\Lambda^{o}\mathcal{Y}_{2n}
&:=
\mathcal{Y}_{2n+1}-\hat{u}_{2n+1}\mathcal{Y}_{2n}-
\hat{\lambda}_{2n+1}\mathcal{Y}_{2n-1},
\\
\Lambda^{o}\mathcal{Y}_{2n+1}
&:=
\mathcal{Y}_{2n+2}-\hat{u}_{2n}\mathcal{Y}_{2n+1}-
\hat{\lambda}_{2n}\mathcal{Y}_{2n},
\end{split}
\end{align}
respectively. We proceed to find the parameters
used in
\eqref{eqn: operator S hat definition phi-2n+1 case}
and
\eqref{eqn: operator T hat definition-2n+1 case}
in terms of the
parameters used in the recurrence relations
\eqref{eqn: recurrence relation for p-2n+1-rho-nu-tao}
and
\eqref{eqn: recurrence relation for p-2n-rho-nu-tao}.
For this, we use the operator relations defined in
\eqref{eqn: operator relation K and D phi-2n+1}
for $\varphi_{2n}$ and $\varphi_{2n+1}$.
Similar to $\mathfrak{D}$, let the operator $\mathfrak{K}$
be defined as
\begin{align}
\label{eqn: operator D and K definition-2n+1 case}
\mathfrak{K}\mathcal{Y}_k:=
\Omega(z)(\mathcal{Y}_{k+1}+\eta_j\mathcal{Y}_{k}),
\qquad \mathcal{Y}_k\in\mathcal{L}_j,
\qquad j\geq0,
\end{align}
where $\Omega(z)$ is a constant with respect to
$\mathfrak{K}$.
Then, we have
\begin{align}
\label{parameters for D K operators}
\begin{split}
\hat{\alpha}&=\alpha,
\quad
\hat{\pi}_{2n}= \pi_{2n+1}+\alpha(\zeta_{2n+1}-\eta_{2n+1}),
\quad
\zeta_{2n-1}\hat{\lambda}_{2n}=\eta_{2n+1}\lambda_{2n},
\\
\hat{\beta}_n&=\beta_{n+1},
\quad
\hat{\bar{\alpha}}\bar{\alpha}\hat{u}_{2n}\hat{\nu}_{2n}\zeta_{2n}+
\bar{\alpha}\hat{\lambda}_{2n}
=
 \hat{\bar{\alpha}}\bar{\alpha}u_{2n}\nu_{2n}\eta_{2n+1}+
 \hat{\bar{\alpha}}\lambda_{2n+1},
\\
\hat{\pi}_{2n+1}&=\pi_{2n}+\beta_{n+1}\zeta_{2n}-\beta_n\eta_{2n},
\quad
\zeta_{2n-2}\hat{\lambda}_{2n+1}=\eta_{2n}\lambda_{2n-1},
\\
\hat{u}_{2n}&=u_{2n+1}-\eta_{2n+1}+\zeta_{2n+1},
\hat{u}_{2n+1}=u_{2n}-\eta_{2n}+\zeta_{2n},
\end{split}
\end{align}
where $\pi_j=u_j\nu_j$, $\hat{\pi}_j=\hat{u}_j\hat{\nu}_j$,
and we define $\hat{\beta}_{-1}:=0$.

This implies that the operators $\Gamma^{o}$ and $\Lambda^{o}$
defined in terms of the parameters
$\hat{\beta}_n$ etc. in
\eqref{eqn: operator S hat definition phi-2n+1 case}
and
\eqref{eqn: operator T hat definition-2n+1 case}
are well-defined.
Now, using
\eqref{eqn: operator relation K and D phi-2n+1},
we note
$\vec{\varrho}_{2n+1}$ is an eigenvector with respect to the
operators $\Gamma$ and $\Lambda$
if, and only if,
$\vec{\hat{\varrho}}_{2n+1}$ is an eigenvector with respect to the operators
$\Gamma^{o}$ and $\Lambda^{o}$.
Let $\theta_j$ be an eigenvector of the
generalized eigenvalue problem
$\Gamma\theta_j=\hat{z}\Lambda\theta_j$,
with the eigenvalue
$\hat{z}$, which is equivalent to
$\theta_j$ being a solution
of the recurrence relation
\eqref{eqn: rational functions in Christoffel section}
with $z$ replaced by $\hat{z}$.
Then, we have
$(\Gamma^{o}-\hat{z}\Lambda^{o})
\mathfrak{D}\theta_{2n+1}=0$
which gives
$\zeta_{2n+1}=-\theta_{2n+2}/\theta_{2n+1}$,
$n\geq0$.
A similar argument for $\theta_{2n}$
gives $\mathfrak{D}\theta_{2n}=0$, which implies
$\zeta_{2n}=-\theta_{2n+1}/\theta_{2n}$,
thus completing the proof.
\end{proof}
The expressions for $\eta_j$ are obtained from the operator relations
$\Lambda^{o}\mathfrak{D}\mathcal{Y}_k=
\hat{z}\mathfrak{K}\Lambda\mathcal{Y}_k$
for $\mathcal{Y}_k=\theta_{2n}$ and $\theta_{2n+1}$
as
\begin{align}
\label{eqn: eta definitions for case 2n+1}
\begin{split}
\eta_{2n}
&=
-\frac{\theta_{2n+1}-u_{2n}\theta_{2n}-\lambda_{2n}\theta_{2n-1}}
{\theta_{2n}-u_{2n-1}\theta_{2n-1}-\lambda_{2n-1}\theta_{2n-2}}
\quad\mbox{and}\\
\eta_{2n+1}
&=
-\frac{\theta_{2n+2}-u_{2n+1}\theta_{2n+1}-\lambda_{2n+1}\theta_{2n}}
{\theta_{2n+1}-u_{2n}\theta_{2n}-\lambda_{2n}\theta_{2n-1}}.
\end{split}
\end{align}

In particular, from
\eqref{parameters for D K operators}
the following relations
\begin{align}
\label{eqn: rho-0 nu-0 relation}
\hat{u_0}\hat{\nu_0}\zeta_0+\frac{\hat{\lambda}_0}{\bar{\alpha}}=
u_0\nu_0\eta_1+\frac{\lambda_1}{\bar{\alpha}}
\quad\mbox{and}\quad
\hat{u}_0=u_1+\zeta_1-\eta_1.
\end{align}
hold for $n=0$.
We use the relations \eqref{eqn: rho-0 nu-0 relation} to find the (constant)
$\Omega(z)$ occurring in the definitions of both the operators
$\mathfrak{D}$ and $\mathfrak{K}$
leading to the Christoffel type
transform of $\varphi_{2n+1}(z)$.
We also remark here that though $\beta_0=0$, we continue using $\beta_0$ in the expressions that follow.
The reason is to show explicitly, the role played by $\beta_0$ in the calculations
\begin{theorem}
\label{thm: Christoffel transform for phi-2n+1}
The Christoffel type transform of
$\varphi_{2n+1}(z)$ is given by
\begin{align*}
  \hat{\varphi}_{2n+1}(z)=
  \sigma\frac{z-\alpha_1}{z-\hat{z}}
  \left[
  \varphi_{2n+2}(z)-
  \frac{\varphi_{2n+2}(\hat{z})}{\varphi_{2n+1}(\hat{z})}
  \varphi_{2n+1}(z)
  \right]
\end{align*}
for a constant $\sigma$.
Further if
$\vec{\varrho}=
\left(
\begin{array}{ccc}
\varphi_0 & \varphi_1 & \cdots \\
\end{array}
\right)^T
$
is the eigenvector for the
generalized eigenvalue problem
$\Gamma\vec{\varrho}=z\Lambda\vec{\varrho}$,
there exists another generalized eigenvalue problem
$\Gamma^{o}\vec{\hat{\varrho}}=z\Lambda^{o}\vec{\hat{\varrho}}$,
with the same eigenvalue $z$ for which
$\vec{\hat{\varrho}}=
\left(
\begin{array}{ccc}
\hat{\varphi}_0 & \hat{\varphi}_1 & \cdots \\
\end{array}
\right)^T$
is the eigenvector.
\end{theorem}
\begin{proof}
The last part of the theorem is about the existence of
generalized eigenvalue problems for the column vectors
$\vec{\varrho}$ and $\vec{\hat{\varrho}}$ which follows from the proof of
Lemma \ref{lem: Lemma for operator D}.
It is also clear that the Christoffel type transform is given by the shift operator
$\mathfrak{D}$ and hence
we need to find $\Omega(z)$ which is independent of $n$.
Further, we obtained the functions
$\theta_j$, $j\geq0$, with $\theta_{-1}=0$,
that satisfy the recurrence relations
\eqref{eqn: rational functions in Christoffel section} with $z$ replaced
by $\hat{z}$.
These equations written explicitly are
\begin{subequations}
\begin{align}
\label{eqn: theta-2n+1 recurrence in case of phi-2n+1}
\alpha\theta_{2n+1}-u_{2n}\nu_{2n}\theta_{2n}-
(\lambda_{2n}/\bar{\alpha})\theta_{2n-1}
&=
\hat{z}[\theta_{2n+1}-u_{2n}\theta_{2n}-\lambda_{2n}\theta_{2n-1}],
\\
\beta_{n+1}\theta_{2n+2}-u_{2n+1}\nu_{2n+1}\theta_{2n+1}-
\beta_n\lambda_{2n+1}\theta_{2n}
&=
\hat{z}[\theta_{2n+2}-u_{2n+1}\theta_{2n+1}-\lambda_{2n+1}\theta_{2n}].
\label{eqn: theta-2n recurrence in case of phi-2n+1}
\end{align}
\end{subequations}
Let the Christoffel type transform of $\varphi_{2n+1}(z)$
be the rational function
\begin{align*}
\hat{\varphi}_{2n+1}(z)=
\frac{\hat{r}_{2n+1}(z)}
{(z-\hat{\alpha})^{n+1}\prod_{j=1}^{n}(1-z\hat{\beta}_j)}
=
\frac{\hat{r}_{2n+1}(z)}
{(z-\alpha)^{n+1}\prod_{j=2}^{n+1}(1-z\hat{\beta}_j)},
\end{align*}
where $\{\hat{r}_{j}(\lambda)\}$ satisfies
\eqref{eqn: recurrence relation for p-2n+1-rho-nu-tao}
and
\eqref{eqn: recurrence relation for p-2n-rho-nu-tao},
but with the coefficients  $u$ replaced by $\hat{u}$ etc.
To determine the constant $\Omega(z)$,
we note that the implication
\begin{align*}
\hat{\varphi}_{2n+1}=\Omega(z)(\varphi_{2n+2}+\zeta_{2n+1}\varphi_{2n+1})
\Longrightarrow
\Omega(z)=\frac{(z-\beta_1)\hat{r}_1(z)}{r_2(z)+\zeta_1(z-\beta_1)r_1(z)}
\end{align*}
follows from the values for $n=0$.
Further, we obtain
$\frac{\theta_1}{\theta_0}=\frac{u_0(\hat{z}-\nu_0)}{\hat{z}-\alpha}$
and
$(\hat{z}-\beta_1)\frac{\theta_2}{\theta_1}=
u_1(\hat{z}-\nu_1)+
\lambda_1(\hat{z}-\beta_0)\frac{\theta_0}{\theta_1}$
from
\eqref{eqn: theta-2n+1 recurrence in case of phi-2n+1}
and
\eqref{eqn: theta-2n recurrence in case of phi-2n+1} for $n=0$
respectively.

Then, Lemma \ref{lem: Lemma for operator D} yields
\begin{align*}
-\zeta_1=\frac{\theta_2}{\theta_1}=
\frac{u_1(\hat{z}-\nu_1)}{\hat{z}-\beta_1}+
\frac{\lambda_1(\hat{z}-\beta_0)(\hat{z}-\alpha)}
{u_0(\hat{z}-\beta_1)(\hat{z}-\nu_0)},
\end{align*}
so that the denominator of $\Omega(z)$ has the expression
\begin{align*}
r_2(\lambda)+\zeta_1(z-\beta_1)r_1(z)
=
\lefteqn{u_0u_1(z-\nu_0)(z-\nu_1)+
\lambda_1(z-\beta_0)(z-\alpha)}\\
&-
\frac{z-\beta_1}{\hat{z}-\beta_1}u_0(z-\nu_0)
\left[
u_1(\hat{z}-\nu_1)+
\frac{\lambda_1(\hat{z}-\beta_0)(\hat{z}-\alpha)}{u_0(\hat{z}-\nu_0)}
\right].
\end{align*}
Further simplification yields
\begin{align*}
\Omega(z)=
\frac{z-\beta_1}{z-\hat{z}}\frac{(\hat{z}-\nu_0)(\hat{z}-\beta_1)\hat{r}_1(z)}
{\Upsilon(z)},
\end{align*}
where $\Upsilon(z) = \Upsilon_1 z+\Upsilon_0$, with
\begin{align*}
\Upsilon_1
&=
u_0u_1(\hat{z}-\nu_0)(\nu_1-\beta_1)+
\lambda_1(\beta_1\nu_0+\beta_0\hat{z}+\alpha\hat{z}-\beta_1\hat{z}-\nu_0\hat{z}-\alpha\beta_0),
\\
\Upsilon_0
&=
-u_0u_1(\hat{z}-\nu_0)(\nu_1-\beta_1)\nu_0+
\lambda_1[\nu_0(\beta_1\hat{z}-\alpha\beta_1-\beta_0\beta_1+\alpha\beta_0)-
\alpha\beta_0(\hat{z}-\beta_1)].
\end{align*}
Next, using the relations
\eqref{eqn: eta definitions for case 2n+1}
and
\eqref{eqn: rho-0 nu-0 relation},
we have $\hat{r}_1(z)=\hat{u}_0(z-\hat{\nu}_0)$,
where
$\hat{u}_0=u_1+\zeta_1-\eta_1$.
Further,
$\hat{u}_0\hat{\nu}_0=u_1+\alpha(\zeta_1-\eta_1)$,
which implies $\hat{u}_0(\alpha-\nu_0)$
\begin{align*}
=\frac{u_1(\nu_1-\beta_1)(\alpha-\hat{z})}{(\hat{z}-\beta_1)}
+
\frac{\lambda_1(\hat{z}-\alpha)}{u_0(\hat{z}-\nu_0)(\hat{z}-\beta_1)}
\left[
\beta_1\hat{z}-\beta_1\nu_0-\beta_0\hat{z}-\alpha\hat{z}+\nu_0\hat{z}+\alpha\beta_0
\right],
\end{align*}
which on further simplification yields
\begin{align*}
\lefteqn{\zeta_0\hat{u}_0(\alpha-\nu_0)(\hat{z}-\beta_1)}
\\
&&=
u_0u_1(\nu_1-\beta_1)(\hat{z}-\nu_0)+
\lambda_1(\beta_1\nu_0+\beta_0\hat{z}+\alpha\hat{z}-\beta_1\hat{z}-\nu_0\hat{z}-\alpha\beta_0).
\end{align*}
Using the fact that
$-\zeta_0=u_0(\hat{z}-\nu_0)/(\hat{z}-\alpha)$,
we have
$\zeta_0\hat{u}_0(\alpha-\nu_0)(\hat{z}-\beta_1)=\Upsilon_1$.
Further, substituting the value of $\eta_1$,
we have from the first relation in
\eqref{eqn: rho-0 nu-0 relation}
\begin{align*}
u_0u_1(\beta_1-\nu_1)(&\hat{z}-\nu_0)
+\lambda_1\left[\nu_0(\hat{z}-\alpha)(\beta_1-\beta_0)-\frac{1}{\bar{\alpha}}(\alpha-\nu_0)(\hat{z}-\beta_1)\right]\\
&+\frac{\hat{\lambda}_0}{\bar{\alpha}}(\alpha-\nu_0)(\hat{z}-\beta_1)
=-\zeta_0(\alpha-\nu_0)(\hat{z}-\beta_1)\hat{\rho}_0\hat{\nu}_0.
\end{align*}
Then, defining $\hat{\lambda}_0:=\lambda_0-\beta_0\bar{\alpha}$
(since $\beta_0=0$, $\hat{\lambda}_0:=\lambda_0$) yields
$-\zeta_0(\alpha-\nu_0)(\hat{z}-\beta_1)\hat{u}_0\hat{\nu}_0=\Upsilon_0$.
Hence, we have
$\zeta_0(\alpha-\nu_0)(\hat{z}-\beta_1)\hat{r}_1(z)=\Upsilon(z)$,
which means
\begin{align*}
\Omega(z)=\frac{\hat{z}-\nu_0}{\zeta_0(\alpha-\nu_0)}\frac{z-\beta_1}{z-\hat{z}}=
\sigma\frac{z-\beta_1}{z-\hat{z}},
\end{align*}
where $\sigma=(\hat{z}-\alpha)/(u_0(\nu_0-\alpha))$.
Finally, we note that since $\theta_j$ satisfies
\eqref{eqn: theta-2n+1 recurrence in case of phi-2n+1}
and
\eqref{eqn: theta-2n recurrence in case of phi-2n+1},
$\theta_j$
must necessarily be equal to $\varphi_j(\hat{z})$.
\end{proof}
\begin{remark}
\label{remark: difference in Christoffel type transform}
We would like to emphasize here the use of the relations
\eqref{eqn: rho-0 nu-0 relation} and the second degree polynomial $r_2(z)$
in deriving the above expressions.
This is different from the one given in
Zhedanov~{\rm\cite{Zhedanov-biorthogonal-GEP-JAT-1999}},
where
the linear polynomial $r_1(z)$ is used.
\end{remark}

We now consider the case $\varphi_{2n}(z)$.
Let
$\hat{\varphi}_{2n}(z)$ denote the Christoffel type
transform of $\varphi_{2n}(z)$, $n\geq0$.
In the present case, we use the shift operators
$\Gamma^{e}$ and $\Lambda^{e}$
where, for $n\geq0$, $\Gamma^{e}$ is given by
\begin{align}
\label{eqn: shift operators Gamma hat Lambda hat phi-2n}
\begin{split}
\Gamma^{e}\mathcal{Y}_{2n}
&:=
\hat{\beta}_{n+1}\mathcal{Y}_{2n+1}-
\hat{u}_{2n+1}\hat{\nu}_{2n+1}\mathcal{Y}_{2n}-
\hat{\beta}_{n-1}\hat{\lambda}_{2n+1}\mathcal{Y}_{2n-1},
\\
\Gamma^{e}\mathcal{Y}_{2n+1}
&:=
\hat{\alpha}\mathcal{Y}_{2n+2}-
\hat{u}_{2n}\hat{\nu}_{2n}\mathcal{Y}_{2n+1}-
\hat{\lambda}_{2n}/\hat{\bar{\alpha}}\mathcal{Y}_{2n},
\end{split}
\end{align}
and $\Lambda^{e}$ is same as $\Lambda^{o}$, which was defined in the case of
$\varphi_{2n+1}(z)$.
The derivation of the expression for
$\hat{\varphi}_{2n}(z)$
follows the same technique as in the case of
$\hat{\varphi}_{2n+1}(z)$.
In fact, this technique is used to find the
Christoffel type transforms of
orthogonal rational functions with arbitrary poles.
However, as remarked earlier, only
the polynomial $r_1(z)$ is used
which makes the calculations easier.
We state only the result for this case.
\begin{theorem}
\label{thm: Christoffel transform for phi-2n}
The Christoffel type transform of $\varphi_{2n}(z)$ is given by
\begin{align*}
\hat{\varphi}_{2n}(z)=
\sigma\frac{z-\alpha}{z-\hat{z}}
\left[
\varphi_{2n+1}(z)-\dfrac{\varphi_{2n+1}(\hat{z})}
{\varphi_{2n}(\hat{z})}
\varphi_{2n}(z)
\right],
\end{align*}
for some constant
$\sigma=(\hat{z}-\alpha)/(u_0(\nu_0-\alpha))$.
Moreover, if
$\vec{\varrho}=
\left(
  \begin{array}{ccc}
    \varphi_0 & \varphi_1 & \cdots \\
  \end{array}
\right)^{T}
$
is the eigenvector for the
generalized eigenvalue problem
$\Gamma\vec{\varrho}=z\Lambda\vec{\varrho}$,
there exists another generalized eigenvalue problem
$\Gamma^{e}\vec{\hat{\varrho}}=
z\Lambda^{e}\vec{\hat{\varrho}}$,
with the same eigenvalue $z$ for which
$\vec{\hat{\varrho}}=\left(
\begin{array}{ccc}
\hat{\varphi}_0 & \hat{\varphi}_1 & \cdots \\
\end{array}
\right)^{T}
$
is the eigenvector.
\end{theorem}

\begin{note}
The constant $\sigma$ is same in both the cases of Christoffel type transforms
of $\varphi_{2n}(z)$ and
$\varphi_{2n+1}(z)$.
\end{note}
We conclude this section with information on the moment functionals associated
with the Christoffel type transforms.
Define the following two linear functionals as
\begin{align}
\label{eqn: moment functionals M and M hat relation}
\mathfrak{N}_0:=\frac{z-\hat{z}}{z-\beta_1}\mathfrak{N}
\quad\mbox{and}\quad
\mathfrak{N}_e:=\frac{z-\hat{z}}{z-\alpha}\mathfrak{N},
\end{align}
where $\mathfrak{N}$ is as defined in
Theorem~\ref{thm: orthogonality result similar to Ismail}.
Further, by multiplication of a functional
by a function $\mathfrak{f}(z)\mathfrak{N}$
it is understood that $\mathfrak{N}$
acts on the space of the space of functions
$\mathfrak{g}(z)$ as
$\mathfrak{N}(\mathfrak{f}(z)\mathfrak{g}(z))$.
Then we have
\begin{theorem}
The following orthogonality relations hold
\begin{align*}
\mathfrak{N}_o
\left(
\frac{z^j}{(1-z\bar{\alpha})^n\prod_{k=0}^{n}(z-\beta_k)}\hat{\varphi}_{2n+1}(z)
\right)
&=0,
\quad j=0,1,\cdots, 2n,\\
\mathfrak{N}_e
\left(
\frac{z^j}{(1-z\bar{\alpha})^n\prod_{k=0}^{n-1}(z-\beta_k)}\hat{\varphi}_{2n}(z)
\right)
&=0,
\quad j=0,1,\cdots, 2n-1,
\end{align*}
where $\mathfrak{N}_0$ and $\mathfrak{N}_e$ are defined in
\eqref{eqn: moment functionals M and M hat relation}.
\end{theorem}
\begin{proof}
Using Theorem \ref{thm: orthogonality result similar to Ismail},
it is easy to see that
\begin{align*}
\mathfrak{N}_o
\left(
\frac{z^j \hat{\varphi}_{2n+1}(z)}{(1-z\bar{\alpha})^n\prod_{k=0}^{n}(z-\beta_k)}
\right)
&=
\sigma\mathfrak{N}
\left(
\frac{z^j(\varphi_{2n+2}(z)+\zeta_{2n+1}\varphi_{2n+1}(z))}
{(1-z\bar{\alpha})^n\prod_{k=0}^{n}(z-\beta_k)}
\right)\\
&=
\sigma\mathfrak{N}
(z^j\{(1-z\bar{\alpha})\mathcal{O}_{2n+2}(z)+\zeta_{2n+1}\mathcal{O}_{2n+1}(z)\})\\
&=0, \quad
j=0,1,2,\cdots,2n,
\end{align*}
where $\mathcal{O}_j(z)$ are the rational functions defined in
\eqref{eqn: intermediary rational functions form}.
The proof for the case of $\vec{\hat{\phi}}_{2n}(z)$ is similar and hence omitted.
\end{proof}

\end{document}